\documentclass[11pt]{amsart}
\usepackage{amsmath, amsrefs, amsthm, amssymb}
\usepackage{graphicx}
\usepackage[active]{srcltx}		
\usepackage{pdfsync}					
\newcommand{\dd}{\,{\rm d}}

\newcommand\R{{\mathbb{R}}}

\newcommand\N{{\mathbb{N}}}

\renewcommand\div{{\rm div}}

\newtheorem{theorem}{Theorem}[section]
\newtheorem{proposition}[theorem]{Proposition}

\theoremstyle{definition}

\theoremstyle{remark}
\newtheorem{remark}[theorem]{Remark}

\numberwithin{equation}{section}
\numberwithin{figure}{section}

\begin{document}

\title[Non-solenoidal approximation of NS]
{On a non-solenoidal approximation to the incompressible Navier-Stokes equations}

\author{Lorenzo Brandolese}

\address{L. Brandolese: Universit\'e de Lyon~; Universit\'e Lyon 1~;
CNRS UMR 5208 Institut Camille Jordan,
43 bd. du 11 novembre,
Villeurbanne Cedex F-69622, France.}
\email{Brandolese{@}math.univ-lyon1.fr}
\urladdr{http://math.univ-lyon1.fr/$\sim$brandolese}

\thanks{Supported by the ANR project DYFICOLTI ANR-13-BS01-0003-01}

\date{\today}


\date{\today}

\begin{abstract}
We establish an asymptotic profile that sharply describes the behavior as $t\to\infty$ for solutions 
to a non-solenoidal approximation of the incompressible Navier--Stokes equations introduced by Temam.
The solutions of Temam's model are known to converge to the corresponding solutions of the 
classical Navier--Stokes, e.g., in $L^3_{\rm loc}(\R^+\times\R^3)$, provided $\epsilon\to0$, where
$\epsilon>0$ is the physical parameter related to the artificial compressibility term. 
However, we show that such model is no longer a good approximation of Navier--Stokes for large times: indeed, its solutions can decay much slower as $t\to\infty$ than the corresponding solutions of Navier--Stokes.
%
\end{abstract}

\maketitle

\section{Introduction}
\label{sec:intro}
Motivated by recent studies by 
C. Niche, M.~E. Schonbek \cite{NicS15} and W.~Rusin~\cite{Rus12}, we consider
the following system, proposed by R.~Temam~\cite{Tem68} as an useful model for the effective approximation of solutions to the Navier--Stokes equations:
\begin{equation}
\label{eq:comp-appr}
 \begin{cases}
  \partial_t u^\epsilon -
  \Delta u^\epsilon+u^\epsilon\cdot \nabla u^\epsilon+\frac12 u^\epsilon\div(u^\epsilon)-\frac1\epsilon\nabla \div(u^\epsilon)=0,
 \qquad x\in \R^n, \; t>0 \\
  u(x,0)=u_0(x), \qquad \div(u_0)=0.
 \end{cases}
\end{equation}
Here $\epsilon>0$ is a parameter measuring how much the vector field $u^\epsilon$ is far from being incompressible.
Notice that $\div(u^\epsilon)$, in general, will not be equal to zero for $t>0$. 
This way of approaching a Navier--Stokes flow with vector fields that are not necessarily 
divergence-free has several advantages in numerical simulations, as pointed out in~\cite{Tem68}.
Indeed, using~\eqref{eq:comp-appr} simplifies the discretization procedures, as one does not need to put the divergence-free constraint at each step for the discretized solutions. Another nice feature of this approximation is that it allows to disregard the nonlocal features (the pressure) of the original system.

The construction of global weak solutions to the system~\eqref{eq:comp-appr} and the
convergence problem as $\epsilon\to0$ of these solutions to the corresponding
Leray solutions of the 
classical Navier--Stokes system are successfully addressed in~\cite{Tem68} and more recently in~\cite{Rus12}.

However, the above mentioned papers provide little information on the long time behavior
of solutions of the approximated model.
The asymptotic behavior as $t\to+\infty$ for the problem~\eqref{eq:comp-appr} has been addressed
only recently in~\cite{NicS15}.
Therein, the authors prove, among other things, that if $u_0\in L^2(\R^n)$ and if the solution of the
linearized equation decays to zero in the $L^2$-norm \emph{at slow rates} as $t\to\infty$,
then the solutions to system~\eqref{eq:comp-appr} will behave like the solutions of the heat equation:
this behavior was somehow the expected one because it is in agreement with the asymptotic behavior 
for the standard Navier--Stokes equations.

On the other hand, when the solution of the linearized equation
decays at \emph{faster rates} 
(namely, faster than $\mathcal{O}(t^{-n/4})$ in the $L^2$-norm),
then the analysis of~\cite{NicS15} does not to apply anymore: the reason is explained in Remark~\ref{wrocla} below.
Typically, we encounter such faster decay rates as soon as $u_0\in L^1(\R^n)$. 

For this reason, the problem of the large time behavior of solutions to~\eqref{eq:comp-appr}
arising from $u_0\in L^1(\R^n)\cap L^2(\R^n)$, is essentially open.

One might wonder that when $u_0\in L^1(\R^n)\cap L^2(\R^n)$ the solutions to~\eqref{eq:comp-appr}
could behave quite differently than the solutions of the Navier--Stokes equations as $t\to+\infty$.
Our analysis will prove that this is indeed the case.
One reason for this is that, contrary to the case of the classical Navier--Stokes,
the mean $\int u^\epsilon(\cdot,t)\dd x$
is no longer constant-in-time.
The long time behavior of  $\int u^\epsilon(\cdot,t)\dd x$ is itself an interesting problem to address, as one expects
its limit to play a predominant role in the description of the long-time asymptotics.
These issues are the motivations of the present paper.
The main result in this direction, stated in Section~\ref{sec:fo}, is the asymptotic profile obtained in Theorem~\ref{th:main} and its applications to sharp two-sided decay estimates for the $L^q$-norms.

The main tool for describing the long time behavior of solutions to~\eqref{eq:comp-appr}
will be the construction of asymptotic profiles for solutions of linear integral equations of the form
\begin{equation}
\label{int-mod}
u(t)=M(t)u_0+\int_0^tM(t-s)f(s)\dd s.
\end{equation}
Here $M(t)$ is a convolution operator with a kernel satisfying the same scaling properties as the heat kernel, and $f$ is a given forcing term.
As such, equation~\eqref{int-mod} is a natural generalization of the classical heat equation in the whole
space. The key result for this linear problem is Theorem~\ref{prop:pro-lin}. This theorem is widely applicable to other nonlinear equations, besides~\eqref{eq:comp-appr}.

\section{On a generalized heat equation in $\R^n$}
\label{sec:gen-heat}

Let $M(x,t)\in C^1(\R^n\times(0,+\infty))$ with the same scaling properties as the heat kernel
$E(x,t)$,
namely,
\begin{equation}
 \label{scalingM}
 M(x,t)=t^{-n/2}M(\frac{x}{\sqrt t},1), \qquad x\in\R^n,\;t>0.
\end{equation}
We also assume that $M$ satisfies the uniform-in-time spatial decay estimates
\begin{equation}
\label{decay-x}
|\partial_x^m M(x,t)|\le C_m|x|^{-n-m}, \qquad m=0,1. 
\end{equation}
Estimates~\eqref{decay-x} together with the scaling properties~\eqref{scalingM} imply the uniform-in-space estimates
\begin{equation}
 \label{decay-t}
 |\partial_t^p M(x,t)|\le C_pt^{-(n+2p)/2}, \qquad p=0,1.
\end{equation}
Both decay estimates hold true for the usual heat kernel.
For all $1<q\le\infty$, and $m,p=0,1$ we also deduce
\begin{equation}
 \label{decay-q}
 \| \partial_x^mM(\cdot,t)\|_q=C_{m,q}\,t^{-\frac n2(1-\frac1q)-m/2}, \qquad
 \| \partial_t^pM(\cdot,t)\|_q=C_{p,q}\,t^{-\frac n2(1-\frac1q)-p}.
\end{equation}

Next theorem describes the long time behavior of solutions to~\eqref{int-mod} starting from vanishing initial data.

\begin{theorem}\mbox{}\\
\label{prop:pro-lin}
\begin{enumerate}
\item
\label{item:part1}
Let $n\ge1$,  $f\in L^1(\R^n\times\R^+)$, with\,  $\|f(t)\|_1=\mathcal{O}(\frac1t)$ as $t\to+\infty$.
Let us introduce the constant 
$ \lambda=\int_0^\infty\!\!\int f(y,s)\dd y\dd s$
and let also
\begin{equation*}
 \label{integralphi}
 \Phi(x,t)=\int_0^t\!\!\int M(x-y,t-s)f(y,s)\dd y\dd s.
\end{equation*}
Then, as $t\to+\infty$,
\begin{equation}
\label{pro:Phi}
 \biggl\| \Phi(t) - \lambda  M(\cdot,t)  \,\biggr\|_q= o(t^{-\frac n2(1-\frac1q)}), 
 \qquad \text{with}\quad
  \begin{cases} 
    1<q\le\infty &\text{if $n=1$},\\
    1< q<\textstyle\frac{n}{n-2} &\text{if $n\ge2$}.
  \end{cases}
\end{equation}
In particular, when $\lambda\not=0$, there exist two constants $c_q,c_q'>0$, independent on~$f$, such that, for~$t>\!\!\!>1$,
\begin{equation}
\label{bound:Phi}
 \lambda c_q \, t^{-\frac n2(1-\frac1q)} \le \|\Phi(t)\|_q \le \lambda c'_q\, t^{-\frac n2(1-\frac1q)}.
\end{equation}
\item
\label{item:part2}
The above results~\eqref{pro:Phi}-\eqref{bound:Phi} extend to $q=1$ if we have in addition $M(\cdot,1)\in L^1(\R^n)$, and to
$\frac{n}{n-2}\le q\le \infty$, provided 
$\|f(t)\|_\beta =\mathcal{O}(t^{-(1+\frac n2(1-\frac1\beta))})$ as $t\to+\infty$, for some $\beta$ such that 
$\frac1q\le \frac{1}{\beta}<\frac1q+\frac2n$.
\end{enumerate}
\end{theorem}

The above asymptotic expansion, with $E$ instead of~$M$, provides an asymptotic profile for the solution of the heat equation with forcing $\partial_t u=\Delta u +f$ in~$\R^n$, and zero initial data. In this case,
the result is valid also for $q=1$.

\begin{proof}[Proof of Theorem~\ref{prop:pro-lin}]
We consider the following decomposition, for any $\eta>0$ and a suitable $1/2<a_\eta<1$ to be chosen later.
 \[
 \begin{split}
  \Phi(t)-&\lambda M(\cdot,t)\\
  \qquad
  &=\int_0^t\!\!\int M(x-y,t-s)f(y,s)\dd y\dd s-  M(\cdot,t) \int_0^\infty\!\!\!\int f(y,s)\dd y \dd s\\
  &=I_1+I_2+I_3+I_4,
 \end{split}
 \]
where
\[
\begin{split}
 I_1&=-M(x,t)\int_{a_\eta t}^\infty\!\int f(y,s)\dd y \dd s,\\
 I_2&=\int_0^{a_\eta t}\!\!\int [M(x,t-s)-M(x,t)]f(y,s)\dd y\dd s,\\
 I_3&=\int_0^{a_\eta t}\!\!\int[M(x-y,t-s)-M(x,t-s)]f(y,s)\dd y \dd s,\qquad\text{and}\\
 I_4&=\int_{a_\eta t}^t\!\!M(x-y,t-s)f(y,s)\dd y\dd s.
 \end{split}
\]
Now,
\[
 \|I_1(t)\|_q\le C_q\,t^{-n/2(1-1/q)}\int_{a_\eta t}^\infty\!\int |f(y,s)|\dd y \dd s=o(t^{-\frac n2(1-\frac1q)}),
 \qquad \text{as $t\to+\infty$}.
\]
by the assumption $f\in L^1(\R^+\times\R^n)$ and the dominated convergence theorem.

For the estimate of $I_2$ we make use of the Taylor formula
\[
 M(x,t-s)-M(x,s)=-\int_0^1\partial_tM(x,t-\theta\,s)\cdot s\dd \theta.
\]
For all $0\le s\le a_\eta t$, using estimate~\eqref{decay-q} we get
\[
\begin{split}
 \|M(\cdot, t-s)- M(\cdot,s)\|_q 
 &\le C_{1,q}(t-\theta s)^{-\frac n2(1-\frac1q)-1} s\\
 &\le c_{\eta,q}\, C_{1,q}t^{-\frac n2(1-\frac1q)-1} s.
\end{split}
 \]
Then we get
\[
 \|I_2(t)\|_q\le c_{\eta,q}\, C_{1,q}t^{-\frac n2(1-\frac1q)-1} \int_0^{a_\eta t}s\|f(\cdot,s)\|_1\dd s.
\]
But $\int_0^{a_\eta t}s\|f(\cdot,s)\|_1\dd s=o(t)$ as $t\to\infty$, as it can be checked by splitting the last integral, {\it e.g.\/},
into $\int_0^{\sqrt{a_\eta t}}$ and $\int_{\sqrt{a_\eta t}}^{a_\eta t}$ and then applying the dominated convergence theorem (we only need $f\in L^1(\R^n\times\R^+)$ here).
We thus get 
\[
 \|I_2\|_q=o(t^{-\frac n2(1-\frac1q)}),
 \qquad \text{as $t\to+\infty$}.
\]

In order to estimate the third integral  we make use of the scaling properties~\eqref{scalingM} of~$M$. 
For all $1<q< \infty$ we have
\[
 \begin{split}
  \|I_3(t)\|_q
   &\le \int_0^{a_\eta t}\!\!\int (t-s)^{-n/2}
      \bigl\| M(\textstyle\frac{\cdot\,-y}{\sqrt{t-s}},1)-M(\textstyle\frac{\cdot}{\sqrt{t-s}},1) \bigr\|_q \,|f(y,s)|\dd y\dd s\\
  &\le  c_{\eta,q}\, t^{-\textstyle\frac n2(1-\frac1q)}   \int_0^\infty\!\!\int 
      {\bf 1}_{[0,a_\eta t]}(s)\bigl\| M(\cdot-\textstyle\frac{y}{\sqrt{t-s}},1)-M(\cdot,1)\bigr\|_q \,|f(y,s)|\dd y\dd s,\\
  \end{split}
  \]
where ${\bf 1}_S$ denotes the indicator function of the set~$S$.
The integrand is dominated by the integrable function $2\|M(\cdot,1)\|_q\,|f(y,s)|$. Moreover, for a.e. $(y,s)\in \R^n\times(0,a_\eta t)$,
by the continuity under translations of the $L^q$-norm
(or, when $q=\infty$, by the fact that $M(\cdot,t)$ is uniformly continuous), we have
\[
   {\bf 1}_{[0,a_\eta t]}(s)\bigl\| M(\cdot-\textstyle\frac{y}{\sqrt{t-s}},1)-M(\cdot,1)\bigr\|_q \,f(y,s)\to0, \qquad\text{as $t\to\infty$}.
\]
The dominated convergence theorem then yields 
\[
 \|I_3\|_q=o(t^{-\frac n2(1-\frac1q)}),
 \qquad \text{as $t\to+\infty$}.
\]

Let us now consider $I_4$. 
Applying the Young inequality
we have, for $1\le \alpha,\beta,q\le\infty$,
\begin{equation}
\label{eq:youn}
 \|I_4(t)\|_q\le \int_{a_\eta t}^t \|M(t-s)\|_\alpha \|f(\cdot,s)\|_\beta\dd s, \qquad \textstyle1+\frac1q=\frac1\alpha+\frac1\beta.
\end{equation}

Let us first consider the cases $n=1$, or $n\ge2$ and $1\le q<\frac{n}{n-2}$.
This ensures $(n/2)(1-1/q)<1$.
For any $\eta>0$ small enough, take $a_\eta=1-\eta^{1/(1-(n/2)(1-1/q))}$, in a such way that $1/2<a_\eta<1$.
In these cases we apply the above estimate with $\alpha=q$ and $\beta=1$.
We now make use of the assumption $\|f(t)\|_1=\mathcal{O}(\frac1t)$ to deduce the estimate, for large enough~$t$,
\[
\begin{split}
 \|I_4\|_q  
 &\le C_{q} \int_{a_\eta t}^t (t-s)^{-\frac n2(1-\frac1q)}s^{-1} \dd s\\
 &\le C_q\, t^{-1} \int_{a_\eta t}^t (t-s)^{-\frac n2(1-\frac1q)}\dd s\\
 &=C_q\,\eta\, t^{-\frac n2(1-\frac1q)}.
\end{split}
 \]
As $\eta>0$ is arbitrarily small, the conclusion follows in this case.

When $n\ge 2$ and $\frac{n}{n-2}\le q\le\infty$, we need to take in estimate~\eqref{eq:youn}
$1-\frac2n<\frac{1}{\alpha}\le1$, and so $\frac1q\le\frac{1}{\beta}<\frac1q+\frac2n$.
We consider again an arbitrary small $\eta>0$, but we take now
$a_\eta=1-\eta^{1/(1-(n/2)(1-1/\alpha))}$.
The additional assumption 
$\|f(t)\|_\beta=\mathcal{O}(t^{-(1+\frac n2(1-\frac1\beta))})$ yields the estimate, for large enough~$t$,
\[
\begin{split}
 \|I_4\|_q  
 &\le C_{q} \int_{a_\eta t}^t (t-s)^{-\frac n2(1-\frac1\alpha)}s^{-1-\frac{n}{2}(1-\frac1\beta)} \dd s\\
 &\le C_q\, t^{-1-\frac{n}{2}(1-\frac1\beta)} \int_{a_\eta t}^t (t-s)^{-\frac n2(1-\frac1\alpha)}\dd s\\
 &=C_q\,\eta\, t^{-\frac n2(1-\frac1q)}.
\end{split}
 \]
Hence, the conclusion follows also in this case.
\end{proof}

Calculations in the same spirit as the above were done in~\cite{FujM01}, in the particular 
case of the heat kernel. In fact, in \cite{FujM01} the computations were carried for the Navier--Stokes equations.

Theorem~\ref{prop:pro-lin} is the crucial step for establishing Theorem~\ref{th:main} below.
But the former theorem should be of independent interest, as the scaling methods used for obtaining asymptotic profiles can be applied to other nonlinear problems.

\section{Statement of the main result}
\label{sec:fo}

A weak solution to problem~\eqref{eq:comp-appr} is a vector field 
$u^\epsilon\in L^\infty(0,\infty),L^2(\R^n))\cap L^2((0,\infty), \dot H^1(\R^n))$ such 
that for every $\phi\in \mathcal{D}(\R^n,[0,\infty))$, $\div \phi=0$, we have
\[
\int_0^\infty\!\!\!\int 
\Bigl(\nabla u^\epsilon\cdot \nabla\phi + (u^\epsilon\cdot\nabla u^\epsilon)\cdot \phi
+\frac12 u^\epsilon\cdot \phi\div u^\epsilon-u^\epsilon\cdot\partial_t\phi
\Bigr)\dd x\dd t
=\int u_0\cdot \phi(\cdot,0)\dd x
\]
If $u_0\in L^2(\R^n)$, and $\div u_0=0$, then the existence (and the uniqueness for $n=2$) of a weak
solution to problem~\eqref{eq:comp-appr} was established in~\cites{Rus12, Tem68}. As for the classical
Navier--Stokes equations such solution satisfies the  energy inequality
\begin{equation}
\label{strong:ei}
\int |u^\epsilon(x,t)|^2\dd x+ 2\int_s^t\!\!\!\int \Bigl( |\nabla u^\epsilon(x,r)|^2
+\frac1\epsilon|\div u^\epsilon(x,r)|^2  \Bigr)\dd x\dd r
\le \int|u^\epsilon(x,s)|^2\dd x,
\end{equation}
for $s=0$ and all $t\ge 0$. At least for $2\le n\le 4$, one can also prove the validity of
the so-called \emph{strong energy inequality}, that is the validity of~\eqref{strong:ei} for almost $s>0$  and all $t\ge s$.

In the following, we denote the heat kernel by 
\[
E(x,t)=e^{-|x|^2/(4t)}/(4\pi t)^{n/2}.
\]

The usual $L^q$-estimates for $E$ are 
\[
\|E(\cdot,t)\|_q=c_q t^{-(n/2)(1-1/q)}, \qquad 1\le q\le\infty.
\]
Our main result reads as follows:

\begin{theorem}\mbox{}\nobreak
\label{th:main}
Let $2\le n\le 4$, $\epsilon>0$ and $u_0\in L^1\cap L^2(\R^n)$ be a divergence-free vector field.
Let $u^\epsilon$ be a weak solution to~\eqref{eq:comp-appr}, satisfying the strong energy inequality~\eqref{strong:ei}.
Then, $u^\epsilon$ becomes eventually a strong solution and the limit 
\begin{equation}
\label{liminte}
\vec\lambda_\epsilon\equiv\lim_{t\to+\infty}\int u^\epsilon(x,t)\dd x
\end{equation}
does exist and is finite.
The vector $\vec\lambda_\epsilon$ describes the long time behavior of~$u^\epsilon$ in the following sense:
\begin{equation}
\label{eq:pro-u}
 \biggl\| u^\epsilon(\cdot,t) +  E(\cdot,t)\vec\lambda_\epsilon  \,\biggr\|_q= o(t^{-\frac n2(1-\frac1q)}), 
 \qquad \text{with}\quad
    1\le q\le \infty.
\end{equation}
Moreover, if $\vec\lambda_\epsilon\not=0$, then there exist $c_\epsilon,c'_\epsilon>0$ such that,
for $1\le q\le\infty$,
\begin{equation}
 \label{eq:ul-bounds-u}
 c_\epsilon |\vec \lambda_\epsilon| \,t^{-\frac n2(1-\frac1q)}
 \le \|u^\epsilon(t)\|_q
 \le c'_\epsilon|\vec\lambda_\epsilon| \, t^{-\frac n2(1-\frac1q)}, \qquad\text{for $t>\!\!\!>1$}.
\end{equation}

The above conclusions holds in any dimension $n\ge2$ if $u_0$ belongs to $L^1\cap L^n(\R^n)$
with $\|u_0\|_n$ small enough and $u^\epsilon$ the unique global strong solution to~\eqref{eq:comp-appr}.
The smallness condition on $\|u_0\|_n$ can be replaced by the weaker assumption that 
$\|u_0\|_{\dot B^{-1+n/p}_{p,\infty}}$ is small enough, for some $n<p<2n$.
\end{theorem}

Let us stress the fact that, in general, $\vec\lambda_\epsilon\not=0$ for $\epsilon>0$.
We shall construct in Section~\ref{sub-sec:ex} an explicit example of initial data in the Schwartz class $\mathcal{S}(\R^3)$ such that  
$\vec\lambda_\epsilon\not=(0,0,0)$.
In particular, this means that these solutions of~\eqref{eq:comp-appr} will satisfy 
$\|u^\epsilon(t)\|_q\sim t^{-\frac n2(1-\frac1q)}$
for $\epsilon>0$, when $u_0$ is integrable.
This constrasts with the case $\epsilon=0$ of
the Navier--Stokes equations: indeed, solutions of the Navier--Stokes equations are known  to decay 
as $\|u(t)\|_2\sim t^{-(n+2)/4}$ as soon as $u_0$ is well localized, see \cite{Wie87}, and sometimes
even at faster rates (e.g., under appropriate symmetries). See contribution~\cite{BraS16}
for an up-to-date review of decay issues for the Navier-Stokes flows.


\begin{remark}
\label{wrocla}
Theorem~\ref{th:main} corrects one of the results of the paper~\cite{NicS15}.
Therein, the authors develop the theory of ``decay characters'' and give several applications
of this notion. One application concerns equation~\eqref{eq:comp-appr}.
The authors asserted that, when $n=3$,
$\|u^\epsilon(t)\|_2\le C(1+t)^{-5/4}$, under the assumptions that $u_0\in L^2(\R^3)$ 
is divergence-free and satisfies the moment condition
$\int(1+|x|)|u_0(x)|\dd x<\infty$ (in~\cite{NicS15} this moment condition is in fact replaced by the closely related condition $r^*(u_0)\ge1$, where $r^*(u_0)$ is
the so-called ``decay character'' of $u_0$).
This assertion is in contradiction with our lower bounds in~\eqref{eq:ul-bounds-u} that shows,
for $n=3$ and $q=2$ that in general $\|u(t)\|_2\sim t^{-3/4}$ for large~$t$.
It is possible to fix the proof of \cite{NicS15}*{Theorem~3.9}
(the pointwise inequality therein for $G(\xi,t)$ and the subsequent calculations could be easily 
corrected), but 
at the price of obtaining a weaker result, namely a lower decay rate in the upper-bound estimate.
This could be done using the same ideas as in~\cite{NicS15}, based on the Fourier-splitting technique.
In the present paper we follow however a different approach (somehow inspired by~\cite{MiyS01}) that has the advantage providing an exact asymptotic profile.
\end{remark}

\section{Global strong solutions uniformly integrable in time}
\label{sec:uit}

\subsection{The linearized equation}
The associated linear problem to Eq.~\eqref{eq:comp-appr} is
\begin{equation}
 \begin{cases}
 \partial_t u^\epsilon -
  \Delta u^\epsilon-\frac1\epsilon\nabla \div(u^\epsilon)=g\\
 u(x,0)=u_0(x). 
\end{cases}
\end{equation}
The integral formulation associated with this linear problem reads
\begin{equation}
u^\epsilon(x,t)=M_\epsilon(t)u_0(x)+\int_0^t M_\epsilon(t-s)g(s)\dd s,
\end{equation}
where $M_\epsilon(t)u_0(x)$ is given by the convolution integral
\[
 M_\epsilon(t)u_0(x)=\int M_\epsilon(x-y,t)u_0(y)\dd y.
\]
The properties of the kernel~$M_\epsilon(x,t)$ have been studied in detail by W.~Rusin in~\cite{Rus12}.
Its symbol is
\begin{equation}
\label{symbM}
 (\widehat M_\epsilon(\xi,t))_{k,l}
 =e^{-t|\xi|^2}\biggl(\delta_{k,l}-\frac{\xi_k\xi_l}{|\xi|^2}\bigl(1-e^{-t|\xi|^2/\epsilon}\bigr)\biggr).
\end{equation}

We complement here the analysis in~\cite{Rus12} of the kernel~$M_\epsilon(t)$ by observing that,
for $t,\epsilon>0$, the kernel $M_\epsilon(\cdot,t)$ belongs to the Schwartz class.
Indeed, we can write
\[
(\widehat M_\epsilon(\xi,t))_{k,l}=e^{-t|\xi|^2}(\delta_{k,l}+\textstyle\frac{t}{\epsilon}\,
\xi_k\xi_l\, H(-t|\xi|^2/\epsilon)),
\]
where $H(r)=(e^{-r}-1)/r=\sum_{k=0}^\infty (-r)^k/(k+1)!$.
This expressions shows that the map $\xi\mapsto (\widehat M_\epsilon(\xi,t))_{k,l}$ can be extended
smoothly in a neighborhood of $\xi=0$, and it is simple to check that its derivatives of any order
decay exponentially to zero as $|\xi|\to\infty$. 

We also observe that
\begin{equation}
\div u=0\quad\Rightarrow\quad
M_\epsilon(t) u=e^{t\Delta}u.
\label{ideM1}
\end{equation}
Another useful identity is the following, valid for any vector field~$f$ 
(not necessarily divergence-free):
\begin{equation}
\label{ideM2}
\div(M_\epsilon(t)f)=e^{t(1+1/\epsilon)\Delta}(\div f).
\end{equation}

It readily follows from the above expressions of the symbol that, for all $\epsilon>0$,
\[ M_\epsilon(x,t) = t^{-n/2}M_\epsilon(x/\sqrt t ,1).\]
Moreover, from these scaling properties and the fact that $M(\cdot,t)$ belongs to the Schwartz class, it follows that, for $t>0$,
\[
 |\partial_x^m M_\epsilon(x,t)| \le C_m|x|^{-n-m}, \qquad m\in \N,
\]
where $C_m>0$ is independent on $x$ and $t$.
Therefore, $M_\epsilon$ satisfies all the estimates~\eqref{scalingM}--\eqref{decay-q}, including 
for $q=1$.

\subsection{The iterative scheme}

The problem~\eqref{eq:comp-appr} can now be conveniently reformulated in the following
integral form:
\begin{equation}
\label{abstr:eq}
u^\epsilon(t)=u_1(t)+B(u^\epsilon,u^\epsilon), 
\qquad u_1(t)=M_\epsilon(t)u_0,\qquad
\qquad \div(u_0)=0.
\end{equation}

Here $B$ is the bilinear operator
\begin{equation}
\label{eq:210}
B(u,v)=-\int_0^t M_\epsilon(t-s)[u\cdot \nabla v+\frac12 u\,\div{v}](s)\dd s.
\end{equation}

The considerations below apply to any abstract equation of the form $u=u_1+B(u,u)$:
let $\mathcal{F}$ be a Banach space, $u_1\in\mathcal{F}$ and let
$B\colon\mathcal{F}\times\mathcal{F}\to\mathcal{F}$ be a continuous bilinear operator,
with operator norm~$\|B\|$.
Let us introduce the nonlinear operators $T_k\colon\mathcal{F}\to\mathcal{F}$, $k=1,2\ldots$, defined
by induction through the formulae
\begin{equation*}
 \label{indu}
\begin{split}
 &T_1={\rm Id}_{\mathcal{F}}\\
& T_k(v)\equiv\sum_{l=1}^{k-1}B(T_l(v),T_{k-l}(v)), \quad k\ge2.
\end{split}
\end{equation*}
Notice that the operator $T_k$ is 
the restriction to the diagonal of a $k$-multilinear operator
defined from $\mathcal{F}^k= \mathcal{F}\times\cdots\times \mathcal{F}$ to $\mathcal{F}$.
The following estimate holds true, see~\cite{AuscT}, \cite{Lem02}:
\begin{equation}
 \label{ATE}
\|T_k(u_1)\|_{\mathcal{F}}\le \frac{C}{\|B\|}k^{-3/2}\bigl(4\,\|B\|\,\|u_1\|_{\mathcal{F}}\bigr)^k.
\end{equation}
Thus, under the smallness assumption
\begin{equation}
 \label{absmall}
\|u_1\|_{\mathcal{F}}\le 1/(4\|B\|),
\end{equation}
the series
\begin{equation}
\label{snc}
\Psi(u_1)\equiv\sum_{k=1}^\infty T_k( u_1 ),
\end{equation}
is absolutely convergent in~$\mathcal{F}$
and its sum $\Psi(u_1)$ is a solution
of the equation
$ u=u_1+B(u,u)$.
Furthermore, $\Psi(u_1)$ is the only solution in the closed ball
$\overline{B_{\mathcal{F}}}(0,\frac{1}{2\|B\|})$ (see~\cite{AuscT}, \cite{Lem02}).

\medskip

Coming back to our model~\eqref{eq:comp-appr}, recalling $u_1(t)=M_\epsilon(t)u_0$,
we would like to establish the existence and the uniqueness of a solution
in a suitable functional setting and to write it as 
$u=\Phi(u_0)$, where
$\Phi(u_0)(t)=\Psi(u_1)(t)=\sum_{k=1}^\infty T_k(u_1)$.
For later use, we will need the series being absolutely convergent in $L^\infty([0,\infty),L^1(\R^d))$.
There are several ways to achieve this: the quickest way 
is to  choose $\mathcal{F}$ to be an appropriate subspace of $L^\infty([0,\infty),L^1(\R^d))$.
A good choice is the following: we define $\mathcal{F}$ to be the Banach space of all $L^\infty([0,\infty),L^1(\R^d))$ functions
such that $\|f\|_{\mathcal{F}}<\infty$, where
\begin{equation}
\begin{split}
\|f\|_{\mathcal{F}}
&=
\sup_{t>0}\|f(t)\|_1 + \sup_{t>0}\sqrt t\|\nabla_x f(t)\|_1\\
&\qquad\quad
+\sup_{t>0} (1+t)^{n/2}\|f(t)\|_\infty + \sup_{t>0}(1+t)^{(n+1)/2}\|\nabla_x f(t)\|_\infty
<\infty.
\end{split}
\end{equation}

We are now in the position of establishing the following proposition:

\begin{proposition}
\label{prop:lili}
There are two constants $\eta>0$ and $c>0$ such that if
\begin{equation}
\label{sett}
\|u_0\|_1+\|u_0\|_\infty<\eta,
\end{equation}
then there is a solution $u^\epsilon\in\mathcal{F}$ of equation~\eqref{eq:comp-appr}, such that
\begin{equation}
\label{sols2}
u^\epsilon=\Phi(u_0)(t)\equiv\sum_{k=1}^\infty T_k\bigl(M_\epsilon(t)u_0\bigr),
\end{equation}
belonging to and unique in the ball
$\{u\in \mathcal{F}\colon \|u\|_{\mathcal{F}}\le c\eta\}$.
The series~\eqref{sols2} is absolutely convergent in the~$\mathcal{F}$-norm.
In particular,
$u^\epsilon$ is uniformly integral in time and, for all $t\ge0$,
\begin{equation}
 \label{integral}
 \int u^\epsilon(x,t)\dd x=\sum_{k=3}^\infty \int T_k\bigl(M_\epsilon(t)u_0\bigr)\dd x.
 \end{equation} 
Moreover, the limit
\[\lim_{t\to+\infty} \int u^\epsilon(x,t)\dd x\]
does exist and is finite.
\end{proposition}

The smallness condition~\eqref{sett} is somewhat unpleasant ---one usually express smallness conditions
in scaling invariant norm, like the $L^n$-norm or even weaker norms--- and can indeed be relaxed.
We will indicate in Section~\ref{sec:relax} how to obtain global solutions belonging to $\mathcal{F}$ replacing condition~\eqref{sett} by a much weaker condition on a scaling invariant
Besov space, at the price of needing more involved bilinear estimates.

\begin{proof}
The bicontinuity of the bilinear operator~$B$ is easily proved in
this space~$\mathcal{F}$.
Indeed, the two $L^1$-estimates for $B(u,v)$ and $\nabla B(u,v)$
follow easily applying to $M_\epsilon$ the first of~\eqref{decay-q}
with $q=1$ and $m=0,1$:
this gives
\[
\|B(u,v)(t)\|_1  
\le \int_0^t \|u(s)\|_1\|\nabla v(s)\|_\infty\dd s\lesssim \|u\|_{\mathcal{F}}\|v\|_{\mathcal{F}}
\]
and
\[
\sqrt{t}\,
\|\nabla B(u,v)(t)\|_1
\le
\sqrt{t}
\int_0^t(t-s)^{-1/2}\|u(s)\|_1\|\nabla v(s)\|_\infty
\lesssim \|u\|_{\mathcal{F}} \|v\|_{\mathcal{F}}.
\]
We obtain the two $L^\infty$ estimates for $B(u,v)$ and $\nabla B(u,v)$ 
by splitting the integral at~$t/2$: 
for the integral $\int_0^{t/2}\ldots$ we apply
to $M_\epsilon$ the first of~\eqref{decay-q}
with $q=\infty$ and $m=0,1$; for the integral $\int_{t/2}^\infty\ldots$
we apply to $M_\epsilon$ the first of~\eqref{decay-q} with $q=1$ and $m=0,1$.
This implies the required estimate
\[
\|B(u,v)\|_{\mathcal{F}}\lesssim C\|u\|_{\mathcal{F}}\|v\|_{\mathcal{F}}.
\]

On the other hand, if $u_0\in L^1\cap L^\infty(\R^n)$, and if $u_0$ is divergence-free, then
$M_\epsilon(t)u_0\in \mathcal{F}$, because $M_\epsilon(t)u_0$ agrees with $e^{t\Delta}u_0$.
In fact, by the standard heat kernel estimates:
\[
\|M_\epsilon(t)u_0\|_{\mathcal{F}}\lesssim\|u_0\|_1+\|u_0\|_\infty.
\]
As discussed before, the above estimates imply that, if $\eta>0$ is small enough, then
the series~\eqref{sols2} converges in the $\mathcal{F}$-norm.
Its sum is the unique solution of equation~\eqref{eq:comp-appr}
in a ball of~$\mathcal{F}$ centered at zero and with small radius.
This series converges in particular in the $L^\infty([0,\infty),L^1(\R^n))$-norm, so we can integrate~\eqref{sols2} and
exchange the integral and summations symbols.
Hence, $\int u^\epsilon(x,t)\dd x=\sum_{k=1}^\infty \int T_k\bigl(M_\epsilon(t)u_0\bigr)\dd x$.

To deduce~\eqref{integral} it only remains to check that 
$\int u_1(t)\dd x=0$ and $\int T_2(u_1)\dd x=0$.
But
\begin{equation}
\label{vanin}
\
\int u_1(t)\dd x=
\int M_\epsilon(t)u_0(x)\dd x= \int u_0(x)\dd x=0,
\end{equation}
because $u_0$ is integrable and divergence-free (the last inequality is well known and easy 
to check using the Fourier transform).
Moreover,
\[
\begin{split}
\int T_2(u_1)\dd x 
&=\int B(u_1,u_1)(t)\dd x\\
&=-\int_0^t\!\!\!\int M_\epsilon(t-s)[u_1\cdot \nabla u_1+\textstyle\frac12\div(u_1)u_1]\dd x\dd s\\
&=-\int_0^t\!\!\!\int [u_1\cdot \nabla u_1+\textstyle\frac12\div(u_1)u_1]\dd x\dd s\\
&=\frac12\int_0^t\!\!\!\int [\div (u_1)u_1](s)\dd x\dd s=0,
\end{split}
\]
because, by~\eqref{ideM2}, 
\[
\div(u_1(s))=M_\epsilon(s)\div(u_0)=0.
\]

Let us now discuss the existence of the limit $\lim_{t\to+\infty} \int u^\epsilon(t)\dd x$.
First of all, observe that
\[
\int T_k(u_1)(t)\dd x
=-\sum_{l=1}^{k-1} 
 \int_0^t\!\!\!\int 
  [T_l(u_1)\cdot \nabla T_{k-l}(u_1)+\textstyle\frac12\div(T_l(u_1))T_{k-l}(u_1)]\dd x\dd s.
\]
Moreover, using the definition of the $\mathcal{F}$-norm we get  
\[
\|T_l(u_1)(s)\|_2\|\nabla T_{k-l}(u_1)(s)\|_2\le c_{l,k}(1+s)^{-(n+1)/2},
\]
so, by Schwarz inequality, the integrand in the last equality is in $L^1(\R^+\times \R^n)$.
Hence the limit
\[ 
\ell_k=\lim_{t\to+\infty}\int T_k(u_1)(t)\dd x
\]
does exist and equals
$-\sum_{l=1}^{k-1}\int_0^\infty\!\!\int 
  [T_l(u_1)\cdot \nabla T_{k-l}(u_1)+\textstyle\frac12\div(T_l(u_1))T_{k-l}(u_1)]\dd x\dd s$.
Moreover, $|\ell_k|\le \|T_k(u_1)\|_{\mathcal{F}}$ and~\eqref{ATE} ensures
that the series $\sum_{k=3}^\infty \ell_k$ converges under our smallness assumption on~$u_0$.
On the other hand, the convergence of series~\eqref{sols2} in the $L^\infty(\R^+,L^1(\R^n))$-norm
allows us to exchange in~\eqref{integral} the limit as $t\to+\infty$ with the summation, leading to
$\lim_{t\to+\infty} \int u^\epsilon(t)\dd x=\sum_{k=3}^\infty \ell_k$.
\end{proof}

\section{Long time behavior of $\int u^\epsilon(x,t)\dd x$: an explicit example with $\vec\lambda_\epsilon\not=0$}
\label{sec:example}

This section can be skipped on a first reading. Its goal is to prove that in general $\vec\lambda_\epsilon\not=0$.
This is interesting to appreciate the relevance of the statement of 
Theorem~\ref{th:main}, in particular the interest of the two-sided bound~\eqref{eq:ul-bounds-u}. But the computations of this section are not needed for proving
Theorem~\ref{th:main} itself.

\subsection{The first term in the expansion of $\int u^\epsilon(t)\dd x$.}

We will need an explicit formula for the first term in the right-hand side of~\eqref{integral}.
\begin{equation*}
\begin{split}
\int &T_3(u_1)\dd x
=\int B(u_1,T_2(u_1))\dd x+\int B(T_2(u_1),u_1)\dd x\\
&=-\int_0^t\!\!\!\int  M_\epsilon(t-s)
	\Bigl[u_1\cdot\nabla T_2(u_1)+\textstyle\frac12 u_1\,\div(T_2(u_1))+\frac12T_2(u_1)\div(u_1))
+T_2(u_1)\cdot\nabla u_1\Bigr]\dd s\dd x\\
&=-\int_0^t\!\!\!\int 
	\Bigl[u_1\cdot\nabla T_2(u_1)+\textstyle\frac12 u_1\,\div(T_2(u_1))
+T_2(u_1)\cdot\nabla u_1\Bigr]\dd s\dd x,
\end{split}
\end{equation*}
where have dropped the term $\frac12T_2(u_1)\div(u_1)$ that is identically zero as $\div(u_1)=0$.
But integrating by parts and using again 
$\div(u_1)=0$ shows that $\int u_1\cdot\nabla T_2(u_1)\dd x=0$,
so the first term  inside the integral can also be dropped.
Another integration by parts finally yields,
\begin{equation}
\label{inte3}
\begin{split}
\int T_3(u_1)\dd x
&=\frac12\int_0^t\!\!\!\int
	\bigl[u_1\,\div(T_2(u_1))\bigr](s)\dd s\dd x.
\end{split}
\end{equation}
On the other hand,
\[
T_2(u_1)=B(u_1,u_1)=-\int_0^t M_\epsilon(t-s)(u_1\cdot\nabla u_1(s))\dd s,
\]
and so
\[
\div(T_2(u_1))
=-\int_0^t e^{(t-s)(1+1/\epsilon)\Delta}\div(u_1\cdot \nabla u_1)(s)\dd s,
\]
where  we applied identity~\eqref{ideM2}.
Replacing this expression in the formula~\eqref{inte3}, next using $u_1(t)=M_\epsilon(t)u_0$ 
and~\eqref{ideM1} leads to (we omit the summation on the repeated subscripts):
\begin{equation}
\label{int3}
\begin{split}
\int T_3(u_1)(t)\dd x
  &=-\frac12 \int\!\!\!\int_0^t
  e^{s\Delta}u_0  
  \biggl[\int_0^s e^{(s-\tau)(1+1/\epsilon)\Delta}
  \partial_k\Bigl((e^{\tau\Delta}u_{0,h})(\partial_h e^{\tau\Delta}u_{0,k})\Bigr)\dd \tau\biggr]
  \dd s\dd x.
\end{split}
\end{equation}

We now apply formula~\eqref{int3} to initial data of the following form: $u_0=\eta v_0$,
where $v_0$ is a fixed divergence-free vector field in 
$L^1\cap L^\infty(\R^n)$ and the parameter $\eta>0$ will be chosen small enough, ensuring in this way the
validity of the smallness condition~\eqref{absmall}.
For such choice of $u_0$, we see that the first term in the 
summation appearing in the right-hand side
of~\eqref{integral} satisfies
\[
\int T_3(u_1)(t)\dd x=\eta^3\int T_3(v_1)(t)\dd x, \qquad v_1=M_\epsilon(t)v_0.
\]
The sum of all the other terms of~\eqref{integral} are 
$\mathcal{O}(\eta^4)$ as $\eta\to0$.
Hence, 
\[
\int u^\epsilon(x,t)\dd x=\eta^3\int T_3(v_1)(t)\dd x+\mathcal{O}(\eta^4),
\qquad
\text{as $\eta\to0$}.
\]
Let us choose a divergence-free vector field $v_0\in L^1\cap L^\infty$ such that, for a fixed~$t$,
$\int T_3(v_1)(t)\dd x\not=0$: then it follows that, choosing a $\eta=\eta(t)>0$ small enough,
$\int T_3(u_1)(t)\dd x\not=0$, and so $\int u^\epsilon(t)\dd x\not=0$ by~\eqref{integral}.
In the same way, if we choose $v_0$ in a such way that 
\begin{equation}
\label{liminoz}
\lim_{t\to+\infty}\int T_3(v_1)(t)\dd x\not=0,
\end{equation} 
then, for $u_0=\eta v_0$ and $\eta>0$ small enough, we get
\[
\lim_{t\to+\infty}\int u^\epsilon(t)\dd x\not=0.
\] 
Now, one expects $\int T_3(v_1)(t)\dd x$ not to be zero generically, but this claim should be proved rigorously.
In next subsection, we will content ourselves to construct, by an explicit computation, a simple 
example of a divergence-free vector field~$v_0$ in the Schwartz class such that the non-vanishing limit
condition~\eqref{liminoz} holds.

%

\subsection{Example of initial data such that $\vec\lambda_\epsilon\not=0$.}
\label{sub-sec:ex}
Let construct an initial datum satisfying~\eqref{liminoz}.
Let us consider a divergence-free three-dimensional vector field of the form
\[
v_0=
\begin{pmatrix} -\partial_2 g\\ \partial_1 g\\0\end{pmatrix},
\]
where $g$ is a smooth and well decaying scalar function to be chosen later on.
Expanding the term $\div(v_1\cdot\nabla v_1)$
for such a vector field~$v_0$ (recalling $v_1(\tau)=M_\epsilon(\tau) v_0=e^{t\Delta}v_0$)
we get after a few simplifications
\begin{equation}
\label{compi}
\begin{split}
\div(v_1\cdot\nabla v_1)(\tau)
= 2\Bigl(\partial_1\partial_2 E(\cdot,\tau)*g\Bigr)^2
-2\Bigl(\partial_1^2 E(\cdot,\tau)*g\Bigr)\Bigl(\partial_2^2 E(\cdot,\tau)*g\Bigr).
\end{split}
\end{equation}
We want to choose $g$ in order to make the computations as explicit as possible, 
avoiding to take however a radial function (we cannot put too many symmetries,
otherwise the integral 
$\int T_3(v_1)\dd x$ could vanish).
A good choice will be to take a derivative of the gaussian function, as 
the constants in the 
subsequent computations can be easily performed in this way by a formal calculus program.
For example, let
\[
g(x)=\partial_2E(x,1).
\]
We denote by~$\kappa$ in the following computations a non-zero constant that may change from line to line.
Computing the Fourier transform in~\eqref{compi} gives (the second equality is computer-assisted, the result can be easily checked e.g. with Maple):
\[
\begin{split}
\mathcal{F}\Bigl[
	\div(v_1\cdot\nabla v_1)(\tau)\Bigr](\tau)
&= 
2\mathcal{F}
\Bigl[
\Bigl(\partial_1\partial_2^2 E(\cdot,\tau+1)\Bigr)^2
-\Bigl(\partial_1^2\partial_2 E(\cdot,\tau+1)\Bigr)\Bigl(\partial_2^3 E(\cdot,\tau+1)\Bigr)
\Bigr]\\
&=
  \frac{\sqrt 2}{512\,\pi^{3/2}}\Bigl((\xi_1^2\xi_2^2+\xi_2^4)(\tau+1)-3\xi_1^2-\xi_2^2\Bigr)
\frac{e^{-(\tau+1)|\xi|^2/2}}{(\tau+1)^{7/2}}.
\end{split}
\]
Applying Plancherel theorem to Eq.~\eqref{int3} yields:
\begin{equation}
\label{inee}
\begin{split}
\int &T_3(v_1)(x,t)\dd x\\
&=-\kappa \int_0^t\!\!\!\int_0^s\!\!\!\int 
 \frac{e^{-[(s-\tau)(1+1/\epsilon)+s+(\tau+3)/2]\,|\xi|^2}}{(\tau+1)^{7/2}}
 {\begin{pmatrix}
 \xi_2^2\\-\xi_1\xi_2\\0
 \end{pmatrix}}
 \Bigl((\xi_1^2\xi_2^2+\xi_2^4)(\tau+1)-3\xi_1^2-\xi_2^2\Bigr)
 \dd \xi\dd s \dd \tau
\end{split}
\end{equation}
where $\kappa=\frac{\sqrt 2}{512\,\pi^{3/2}}$.
Let us focus on the first component of the above integral (the second and the third components vanish
for symmetry reasons for all $t>0$).
Letting $\lambda=(s-\tau)(1+1/\epsilon)+s+(\tau+3)/2$ we are led to calculate (with a computer assisted computation) first the integral
in the $\xi$-variable:
\[
\int e^{-\lambda|\xi|^2}\xi_2^2\Bigl((\xi_1^2\xi_2^2+\xi_2^4)(\tau+1)-3\xi_1^2-\xi_2^2\Bigr)\dd\xi
=-\frac{3\,\pi^{3/2}}{4\,\lambda^{9/2}}(2\lambda-3\tau-3).
\]
But $s>\tau$, so $2\lambda-3\tau-3>0$ and it follows that the first component of
$\int T_3(v_1)(x,t)\dd x$ is strictly positive for all $t>0$.
In fact, the first component of the map $t\mapsto \int T_3(v_1)(x,t)\dd x$ is strictly increasing and
so the fist component of the limit $\lim_{t\to\infty} \int T_3(v_1)(x,t)\dd x$ is a strictly
positive real number. 

%

Summarizing, recalling also 
the discussion at the beginning of this section, we 
established that if $\eta>0$ is small enough and 
\[ u_0(x)=\eta\begin{pmatrix}
-\partial_2^2 E(x,1)\\
\partial_1\partial_2 E(x,1)\\
0
\end{pmatrix}
\]
then the solution of~\eqref{eq:comp-appr} starting from~$u_0$ satisfies
\[
\vec{\lambda}_\epsilon=\lim_{t\to\infty}\int u^\epsilon(x,t)\dd x\not=0,
\]
as the first component of~$\vec \lambda$ is strictly positive.

\section{Relaxing the smallness assumption}
\label{sec:relax}

\subsection{Global solutions with small data in Besov spaces}
Inspired by T.~Kato's arguments~\cite{Kat84} for the classical Navier--Stokes equations,
we introduce the following Banach space, for $n\le p<\infty$, where we are going to construct our solutions.
\begin{subequations}
\begin{equation}
\label{Xp}
\mathcal{X}_p=
\{u\in C((0,\infty),L^p(\R^n))\colon
 \|u\|_{\mathcal{X}_p}=\sup_{t>0} t^{(1/2)(1-n/p)}\|u(t)\|_p
+\sup_{t>0} t^{1-n/(2p)}\|\nabla u(t)\|_p<+\infty\}.
\end{equation}
\end{subequations}
Notice that the $\mathcal{X}_p$-norm is left invariant by the natural rescaling 
$u\mapsto u_\lambda=\lambda u(\lambda x,\lambda^2 t)$. Eq:~\eqref{eq:comp-appr} is itself left invariant
by the above rescaling.
If $u_0\in L^n(\R^n)$,  then
the solution of the heat equation $e^{t\Delta}u_0$ belongs to $\mathcal{X}_p$, for $n\le p\le\infty$,
and
\[
\|e^{t\Delta}u_0\|_{\mathcal{X}_p} \le C_p\|u_0\|_n.
\]

The solutions constructed in $\mathcal{X}_p$ will often have faster time decay as $t\to\infty$  than 
predicted by the $\mathcal{X}_p$-norm.
The following space will be useful to describe the decay properties of solutions arising from integrable initial data.
To this purpose, let us introduce,  for $1\le q\le \infty$,
\begin{equation}
\label{Yq}
\begin{split}
\mathcal{Y}_q
=\{u\in C((0,,\infty),L^q(\R^n))\colon
 \|u\|_{\mathcal{Y}_q}=\sup_{t>0}&(1+t)^{(n/2)(1-1/q)}\|u(t)\|_q \\
& +\sup_{t>0} t^{1/2}(1+t)^{(n/2)(1-1/q)}\|\nabla u\|_q<\infty\}.
\end{split}
\end{equation}

First of all, observe that, for $n<p\le\infty$, a tempered distribution in $\R^n$
$u_0$ belongs to the Besov space $ \dot B^{-1+n/p}_{p,\infty}$ if and only if
$t^{(1/2)(1-n/p)} e^{t\Delta}u_0\in L^\infty(\R^+,L^p(\R^n))$ and we have the norm equivalence
\begin{equation}
\label{besov-heat}
\|u_0\|_{\dot B^{-1+n/p}_{p,\infty}}\simeq \sup_{t>0}t^{(1/2)(1-n/p)}\|e^{t\Delta}u_0\|_{p}.
\end{equation}
See~\cite{BahCD11}, \cite{Lem02}.
The scaling of $\dot B^{-1+n/p}_{p,\infty}$ agrees with that of $L^n(\R^n)$ and for $n<p\le \infty$ we have the
inclusion $L^n(\R^n)\subset \dot B^{-1+n/p}_{p,\infty}$.

\begin{proposition}
\label{prop:l1l3}
\mbox{}
\begin{enumerate}
\item[(1)]
Let $n<p<2n$ and $u_0\in \dot B^{-1+n/p}_{p,\infty}$ be a divergence-free vector field. 
There exists $\eta_p>0$ such that if $\|u_0\|_{ \dot B^{-1+n/p}_{p,\infty}}<\eta$ then there is only
one solution $u^\epsilon\in \mathcal{X}_{p}$ to the problem~\eqref{eq:comp-appr}, such that
$\|u^\epsilon\|_{\mathcal{X}_{p}}<2\eta_p$.
Such solution $u^\epsilon$ belongs also to $\mathcal{X}_q$, for $p\le q\le\infty$.
If, more precisely, $u_0\in L^n(\R^n)$, then 
$u^\epsilon$ belongs also to $\mathcal{X}_q$, for $n\le q\le\infty$ and in this case
$u^\epsilon\in C([0,\infty),L^n(\R^n))$. 

\item[(2)]
Under the additional assumption $u_0\in L^1(\R^n)$, then, for all $1\le q\le \infty$,
$u^\epsilon\in \mathcal{Y}_q$.
\end{enumerate}
\end{proposition}

\begin{remark}
\label{rem:ln}
When $u_0\in L^n(\R^n)$ with a small enough $L^n$-norm, then Proposition~\ref{prop:l1l3}~(1)
does apply. This follows from the fact the inclusion map $L^n(\R^n)\subset \dot B^{-1+n/p}_{p,\infty}$,
$n<p\le\infty$, is continuous.  
Proposition~\ref{prop:l1l3}, however allows to construct global solution for some initial data with large $L^n(\R^n)$-norm. Indeed, the smallness condition on the Besov norm $\dot B^{-1+n/p}_{p,\infty}$
will be fullfilled as soon as $u_0$ is fast oscillating. This idea of relaxing the smallness condition
using rough spaces goes back to \cite{Can95}.
\end{remark}

\begin{proof}
The proof of the first part is conceptully close to that in~\cite{Kat84} or~\cite{Can95}, with slight changes in
the choice of the exponents of some relevant estimates. For this reason we will be rather sketchy, omitting in particular to discuss the continuity in the time variable that is standard.
The proof consists in applying the standard fixed point argument in $\mathcal{X}_p$ to the 
integral equation~\eqref{abstr:eq}. 

If $u_0\in  \dot B^{-1+n/p}_{p,\infty}$, then $\nabla u_0\in \dot B^{-2+n/p}_{p,\infty}$
and therefore
\begin{equation}
\label{besov-heat-g}
\sup_{t>0}t^{1-n/(2p)}\|\nabla e^{t\Delta}u_0\|_{p} \lesssim \|u_0\|_{\dot B^{-1+n/p}_{p,\infty}}.
\end{equation}
As $u_0$ is divergence-free, $M_\epsilon(t)u_0=e^{t\Delta}u_0$, as already observed in Eq.~\eqref{ideM1}.
Hence,
\begin{equation}
\label{lines1}
\| M_\epsilon(t)u_0\|_{\mathcal{X}_p}\le C_p \|u_0\|_{\dot B^{-1+n/p}_{p,\infty}}, \qquad n<p\le\infty. 
\end{equation}

The estimates on the bilinear term rely on Young convolution inequality and
the estimates, valid for $1\le r\le\infty$ (see the discussion on $M_\epsilon$ at the end of Subsection~\ref{sec:uit}),
\begin{equation}
\label{est:mt-s}
\|M_\epsilon(t-s)\|_{r} \simeq (t-s)^{-\frac{n}{2}(1-1/r)}, \qquad
\|\nabla M_\epsilon(t-s)\|_{r} \simeq (t-s)^{-1/2 -\frac{n}{2}(1-1/r)}.
\end{equation}
More precisely, choosing $r=p'$ (the conjugate exponent of $p$), with $n<p<2n$, we get
\[
\|B(u,v)(t)\|_p \simeq \int_0^t (t-s)^{-n/(2p)}\|u(s)\|_{p}\|\nabla v(s)\|_{p}\dd s
\lesssim t^{-(1/2)(1-n/p)},
\]
and
\[
\|\nabla B(u,v)(t)\|_p \simeq \int_0^t (t-s)^{-1/2 -n/(2p)}\|u(s)\|_{p}\|\nabla v(s)\|_{p}\dd s
\lesssim t^{-(1-n/(2p))}.
\]
Combining these two estimates we see that, for some constant $C'_p>0$ and all $u,v$:
\begin{equation}
\label{biles1}
\|B(u,v)\|_{\mathcal{X}_p}\le C'_p \|u\|_{\mathcal{X}_p}\|v\|_{\mathcal{X}_p},
\qquad n<p<2n.
\end{equation}
Applying the fixed point lemma (\cite{BahCD11}*{Lemma~5.5}), using estimates~\eqref{lines1} and \eqref{biles1} yields the existence and the unicity of the solution in the ball
$\{u\colon \|u\|_{\mathcal{X}_p}<1/2(C_pC'_p)\}$ of the space $\mathcal{X}_{p}$, 
provided $\|u_0\|_{\dot B^{-1+n/p}_{p,\infty}}<1/(4C_pC'_p)$, and $n<p<2n$.

Let us now prove that such solution belongs to $\mathcal{X}_q$, for all $n<p\le q\le\infty$
(or $n\le q\le\infty$ if we assume $u_0\in L^n(\R^n)$).
We write $u^\epsilon=M_\epsilon(t)u_0+B(u^\epsilon,u^\epsilon)$ and separate the contributions of the
linear and the nonlinear terms.

First of all, by the semigroup properties of the heat kernel, 
$\|M_\epsilon(t)u_0\|_{\mathcal{X}_q}\lesssim \|u_0\|_{\dot B^{-1+n/p}_{p,\infty}}$.
Under the more stringent condition $u_0\in L^n(\R^n)$ the linear term 
satisfies also $\|M_\epsilon(t)u_0\|_{\mathcal{X}_n}\lesssim \|u_0\|_{L^n(\R^n)}$.
Let us prove that the nonlinear term $B(u^\epsilon,u^\epsilon)$
belongs to $\mathcal{X}_n\cap\mathcal{X}_\infty$ (whether or not $u_0\in L^n(\R^n)$).

Applying twice~\eqref{est:mt-s} with $r$ such that $1+1/n=1/r+2/p$ we get, for $n<p<2n$:
\begin{equation*}
\|B(u^\epsilon,u^\epsilon)(t)\|_n
+t^{1/2}\|\nabla B(u^\epsilon,u^\epsilon)(t)\|_n
\lesssim \|u^\epsilon\|^2_{\mathcal{X}_{p}}
\lesssim \|u_0\|_{\dot B^{-1+n/p}_{p,\infty}},
\end{equation*}
and so
\begin{equation}
\label{est:n}
\|u^\epsilon\|_n
+t^{1/2}\|\nabla u^\epsilon\|_n
\lesssim \|u_0\|_{\dot B^{-1+n/p}_{p,\infty}}.
\end{equation}

It only remains to establish the $L^\infty$ estimates for $u^\epsilon$ and 
$\nabla u^\epsilon$.
The former is easy: applying~\eqref{est:mt-s} with $r$ such that $1=1/r+2/p$, we get
\begin{equation*}
t^{1/2}\|B(u^\epsilon,u^\epsilon)(t)\|_\infty \lesssim \|u^\epsilon\|^2_{\mathcal{X}_{p}},
\end{equation*}
implying
\begin{equation}
t^{1/2}\|u^\epsilon(t)\|_\infty \lesssim \|u_0\|_{\dot B^{-1+n/p}_{p,\infty}}.
\end{equation}
We perform the latter in two steps: 
first, applying~\eqref{est:mt-s} for all $n<p\le q<\tilde p<\infty$, with $1+1/q=1/r+1/\tilde p+1/n$, 
\begin{equation*}
\begin{split}
\|\nabla B(u^\epsilon,u^\epsilon)(t)\|_q 
&\le \int_0^t (t-s)^{-1-\frac{n}{2}(1/\tilde p-1/q)}\|u^\epsilon(s)\|_{\tilde p}\|\nabla u^\epsilon(s)\|_n
\\
&
 \lesssim t^{-1+n/(2q)} \|u_0\|_{\dot B^{-1+n/p}_{p,\infty}},
\end{split}
\end{equation*}
where we used $\|u^\epsilon(t)\|_{\tilde p}\lesssim t^{-(1/2)(1-n/\tilde p)}\|u_0\|_{\dot B^{-1+n/p}_{p,\infty}}$, that follows interpolating the two previous estimates on 
$\|u^\epsilon(t)\|_p$ and $\|u^\epsilon(t)\|_\infty$.
Using the above with $q=2n$ we get
\[
\|\nabla u^\epsilon(t)\|_{2n}\lesssim t^{-3/4} \|u_0\|_{\dot B^{-1+n/p}_{p,\infty}}.
\]
Next,
\[
\begin{split}
\|\nabla B(u^\epsilon,u^\epsilon)(t)\|_\infty 
&\le\int_0^{t/2}(t-s)^{-1/2-n/p}\|u^\epsilon(s)\|_p\|\nabla u^\epsilon\|_p\dd s\\
&\qquad\qquad\qquad
+ \int_{t/2}^t (t-s)^{-3/4}\|u^\epsilon(s)\|_\infty\|\nabla u^\epsilon(s)\|_{2n}\dd s\\
&\lesssim t^{-1}\|u_0\|_{\dot B^{-1+n/p}_{p,\infty}}.
\end{split}
\]
We finally get 
$\|\nabla u^\epsilon(t)\|_\infty
\lesssim t^{-1}\|u_0\|_{\dot B^{-1+n/p}_{p,\infty}}$.
Summarizing, $u^\epsilon$ belongs to $\mathcal{X}_p\cap\mathcal{X}_\infty$ and by interpolation 
for all $p\le q\le\infty$ and $n<p<2n$ (with $q=n$ being allowed under the more stringent assumption
$u_0\in L^n(\R^n)$),
\[
\|u^\epsilon\|_{\mathcal{X}_q}\lesssim \|u_0\|_{\dot B^{-1+n/p}_{p,\infty}},
\]
for all divergence-free distribution $u_0$ 
such that $\|u_0\|_{\dot B^{-1+n/p}_{p,\infty}}<1/(4C_pC'_p)=\eta_p$.
This achieves the proof of Part~(1) of Proposition~\ref{prop:l1l3}.

\medskip
Let us now prove Part~(2), of Proposition~\ref{prop:l1l3}.
We have now the additional assumption $u_0\in L^1(\R^n)$.
More in general, we can prove that under the additional assumption $u_0\in L^r(\R^n)$,
with $1\le r<n$, then the solution constructed in Part (1) satisfies the estimates
\begin{equation}
\label{kato:est}
\begin{split} 
 &(1+t)^{(n/2)(1/r-1/q)} u^\epsilon\in L^\infty((0,\infty),L^q(\R^n))
\quad\text{and}
\quad\\
&t^{1/2}(1+t)^{(n/2)(1/r-1/q)} \nabla u^\epsilon\in L^\infty((0,\infty),L^q(\R^n)),
\end{split}
\end{equation}
provided $ \|u_0\|_{\dot B^{-1+n/p}_{p,\infty}}<\eta'$ and $\eta'>0$ is small enough.
Notice that, at least for $1<r<n$ and $r\le q<\infty$ this type of estimates
were proved in \cite{Kat84}*{Theorem~5} for the Navier--Stokes equations.
Moreover, Kato's recursive inequalities method 
provides the norm estimate, for $1<r<n$ and $r\le q<\infty$ 
(and $\epsilon=0$, \emph{i.e.} in the Navier--Stokes case), 
\begin{equation}
\label{normest}
\sup_{t\ge1}  (1+t)^{(n/2)(1/r-1/q)} \|u^\epsilon\|_q
+t^{1/2}(1+t)^{(n/2)(1/r-1/q)} \|\nabla u^\epsilon\|_q 
\lesssim \|u_0\|_r+\eta'. 
\end{equation}
In fact, for $1<r<n$ and $r\le q<\infty$ Kato's proof can be reproduced 
in our case (\emph{i.e.} in the case $\epsilon>0$)
in exactly the same way: indeed the linear $L^r-L^q$ linear estimates 
for $M_\epsilon(t)u_0$ agree with the heat
kernel estimates, and the nonlinearity in Eq.~\eqref{eq:210} has the same structure
as in the integral Navier--Stokes equations.
For this reason we can skip the proof that~\eqref{normest} hold true also 
for solutions of~\eqref{abstr:eq}.
What remains to be done, is to establish~\eqref{kato:est}
in the limit case $r=1$ or $q=\infty$ that were excluded in Kato's approach.
This is possible because, contrary to the case of Navier--Stokes, we do not have to deal with the
projection operator onto the divergence-free vector field, that is unbounded in $L^1$ and $L^\infty$.
In other words, we take advantage of the fact that the kernel of $M_\epsilon(t)$ belongs to 
$L^1\cap L^\infty(\R^n)$.
More precisely, using~\eqref{est:mt-s} with $r=1$ and
applying~\eqref{kato:est} with $r=6/5$ and $q=2$ we get:
\[
\begin{split}
\|u^\epsilon(t)\|_1 
&\le C\|u_0\|_1+\int_0^t \|u^\epsilon\|_2\|\nabla u^\epsilon\|_2\dd s
\le C,\\
\|\nabla u^\epsilon(t)\|_1 
&\le \|u_0\|_1 t^{-1/2}
+\int_0^t (t-s)^{-1/2} \|u^\epsilon\|_2\|\nabla u^\epsilon\|_2\dd s\le  Ct^{-1/2}.
\end{split}
\]
This establishes~\eqref{kato:est} for $r=1$ and $q=1$.
Using the similar arguments one can easily establish the validity of~\eqref{kato:est} for $r=1$ and $q=\infty$.
By interpolation, these estimates are then valid for $r=1$ and $1\le q\le\infty$, hence Part~(2)
of Proposition~\ref{prop:l1l3} follows.

\end{proof}

\subsection{Application to weak solutions}

The uniqueness of weak solutions for $n=2$ for \eqref{eq:comp-appr} is due to~R.Temam~\cite{Tem68}.
When $n\ge3$, the uniqueness of weak solutions is not known. 
In this case, the analogue of classical weak-strong uniqueness result established for the Navier--Stokes equations, like that of Sohr and von Wahl~\cites{Sohv84}
for~\eqref{eq:comp-appr} will be useful.
We state this as a remark:

\begin{remark}
\label{rem:wsu}
Let $u^\epsilon$ and $v^\epsilon$ be two weak solutions as defined in Section~\ref{sec:main},
satisfying the energy inequality~\eqref{strong:ei} (at least in its weak form, \emph{i.e.} for $s=0$)
and let $v^\epsilon\in C([0,T),L^n(\R^n))$, $0<T\le \infty$. Then $u^\epsilon=v^\epsilon$.
\end{remark}

Indeed, writing the equation for the difference $w^\epsilon=u^\epsilon-v^\epsilon$,
formally multiplying by $w^\epsilon$ using the cancellation
\[
\int (u^\epsilon\cdot\nabla) w^\epsilon\cdot w^\epsilon+\frac12\int(w^\epsilon\,\div u^\epsilon)
\cdot w^\epsilon=0,
\]
one obtains from the constructions of the solutions that $w^\epsilon$ satisfies the inequality
\[
\frac12\|w^\epsilon(t)\|_2^2+\int_0^t\!\!\!\int(w^\epsilon\cdot\nabla v^\epsilon)\cdot w^\epsilon
+\frac12\int_0^t\!\!\!\int v^\epsilon(\div w^\epsilon)\cdot w^\epsilon+\frac1\epsilon\int_0^t\|\div w^\epsilon\|_2^2
+\int_0^t\|\nabla w^\epsilon\|_2^2\le 0.
\]
But
\[
\biggl|\int_0^t\!\!\!\int(w^\epsilon\cdot\nabla v^\epsilon)\cdot w^\epsilon
+\frac12\int_0^t\!\!\!\int v^\epsilon(\div w^\epsilon)\cdot w^\epsilon\biggr|
\le C\int_0^t\!\!\!\int |v^\epsilon|\,|w^\epsilon|\, |\nabla w^\epsilon|.
\]
We can now forget about the term 
$\frac1\epsilon\int_0^t\|\div w^\epsilon\|_2^2$ and proceed as for the classical Navier--Stokes:
splitting $v^\epsilon=\bar v^\epsilon+\tilde v^\epsilon$ with $\bar v^\epsilon$ small in the
$L^\infty((0,T),L^n)$-norm and $\tilde v^\epsilon\in L^\infty((0,T),L^\infty)$ and we can absorb the
above term (see \cite{Lem02} for details of these estimates) and get
$\|w^\epsilon(t)\|_2^2+C\|\tilde v^\epsilon\|_{L^\infty_{t,x}}\int_0^t\|w(s)\|_2^2\le0$.
As usual one concludes $w^\epsilon=0$ by Gronwall inequality.

\section{Proof of the main result}
\label{sec:main}

The proof of Theorem~\ref{th:main} relies on construction of weak solutions made in~\cite{Rus12},\cite{Tem68},
on Proposition~\ref{prop:l1l3} and Theorem~\ref{prop:pro-lin}.

\begin{proof}[Proof of Theorem~\ref{th:main}]

Let $2\le n\le 4$ and consider a weak solution $u^\epsilon$ 
to~\eqref{eq:comp-appr} satisfying the strong energy inequality~\eqref{strong:ei}, arising from
the divergence-free vector field $u_0\in L^2(\R^n)$.
Then, for almost all $t_0$, we have $u^\epsilon(t_0)\in H^1(\R^n)\subset L^n(\R^n)$, and we can find
$t_0>0$ such that the energy inequality~\eqref{strong:ei} holds with $s=t_0$ and 
$\|u(t_0)\|_n$ satisfies the smallness assumption mentioned in~Remark~\ref{rem:ln}.
By this remark, there is a strong solution $v^\epsilon\in C([t_0,\infty),L^n(\R^n))$.
As observed in Remark~\ref{rem:wsu}, $u^\epsilon=v^\epsilon$ on $[t_0,\infty)$. So $u^\epsilon$ becomes a strong solution after some time, satisfying for $t\ge t_0$ the conditions of Part~(1) of Proposition~\ref{prop:l1l3}.
But $u^\epsilon$ is also known to satisfy the integral equation~\eqref{abstr:eq},
just like weak solutions of the Navier--Stokes equations satisfying the energy inequality do solve the
corresponding integral equations (we refer to~\cite{Dub02} for a proof this clam in the more general
setting of $L^2_{\rm uloc}$-solutions).
Then the bilinear term $B(u^\epsilon,u^\epsilon)(t)$ in~\eqref{abstr:eq}
belongs to $L^1(\R^n)$.
Of course, under the more stringent condition $u_0\in L^1\cap L^2(\R^n)$ also the linear term $M_\epsilon(t)u_0$ 
in~\eqref{abstr:eq} belongs to $L^1(\R^n)$.
Hence, when $2\le n\le 4$, under the assumptions of Theorem~\ref{th:main}, 
the weak solution considered above, after some time $t_0>0$, satisfies in fact also the 
conditions of Part~(2) of Proposition~\ref{prop:l1l3}.
Namely
\begin{equation}
\label{decay-st}
\begin{split}
&\|u^\epsilon(t)\|_q \le C_q t^{-(n/2)(1-1/q)},\\
&\|\nabla u^\epsilon(t)\|_q \le C_q t^{-1/2-(n/2)(1-1/q)},
\end{split} \qquad t\ge t_0,\quad 1\le q\le\infty.
\end{equation} 

It follows from Schwarz inequality that 
$|u^\epsilon|\,|\nabla u^\epsilon|\in L^1(\R^n\times \R^+)$.
In particular, we can define
\begin{equation}
 \label{eq:lambda-u}
 \vec\lambda_\epsilon\equiv
  \int_{0}^\infty\!\!\!\int \Bigl( u^\epsilon\cdot \nabla u^\epsilon +  \textstyle \frac12 u^\epsilon \div (u^\epsilon)\Bigr)\dd y\dd s.
\end{equation}

Integrating with respect to the space variable equation~\eqref{abstr:eq}
and recalling that $\int M_\epsilon(t)u_0\dd x=0$ (see~\eqref{vanin}), we deduce
\[
\int u^\epsilon(x,t)\dd x=\int_0^t\!\!\!\int \Bigl( u^\epsilon\cdot \nabla u^\epsilon +  \textstyle \frac12 u^\epsilon \div (u^\epsilon)\Bigr)\dd y\dd s,
\]
and so the limit $\lim_{t\to+\infty}\int u^\epsilon(x,t)\dd x$ does exist. Then $\vec\lambda_\epsilon$
is also given by
\begin{equation}
\label{lamin}
 \vec\lambda_\epsilon=\lim_{t\to+\infty}\int u^\epsilon(x,t)\dd x.
\end{equation}
This is in agreement with~\eqref{liminte}.

Consider again the integral formulation~\eqref{abstr:eq}, and let us
start with the study of the asymptotics of the linear part.
As already observed, $M_\epsilon(t)$ boils down to the standard heat kernel when applied to divergence-free vector fields and $\int u_0=0$, by the divergence-free condition and the integrability of $u_0$.
Hence,
\[ 
M_\epsilon(t)u_0(x)=e^{t\Delta}u_0(x)=\int [E(x-y,t)-E(x,t)]u_0(y)\dd y.
\]
Then, as $t\to+\infty$,
\[
\| M_\epsilon(t)u_0(x)\|_q 
\le t^{-\frac n2(1-\frac1q)}\int \textstyle\|E(\cdot-\frac{y}{\sqrt t},1)- E(\cdot,1)\|_q \, |u_0(y)| \dd y
 =o(t^{-\frac n2(1-\frac1q)}).
\]
In the last equality we applied the continuity of the $L^q$-norm under translations (when $1\le q<\infty$)
or the uniform continuity of $E(\cdot,t)$ (when $q=\infty$), and the dominated convergence theorem 
to prove that the last integral goes to zero as $t\to\infty$.
The nonlinear part in~\eqref{abstr:eq} can be written as
\[
 \Phi(\cdot,t)
 =-\int_0^t\!\!\!\int  M_\epsilon(x-y,t-s)f(y,s)\dd x\dd s,
\]
where $f(s)=u^\epsilon\cdot \nabla u^\epsilon +  \textstyle \frac12 u^\epsilon \div (u^\epsilon)$.

Next step will consist in applying Theorem~\ref{prop:pro-lin} with this choice of~$f$ and
the kernel $M=M_\epsilon$. Let us check the validity of the assumptions of our proposition.
We already observed in Section~\ref{sec:uit} that $M_\epsilon(\cdot,t)$ is in $L^1(\R^n)$.
Moreover, applying the decay estimates~\eqref{decay-st} 
$f\in L^1(\R^n\times \R^+)$ and
$t \|f(t)\|_1\in L^\infty(\R^+)$.
Furthermore
combining estimates~\eqref{decay-st} with H\"older inequality we get, for $1\le\beta\le \infty$,
\[
\|f(t)\|_\beta=\|u^\epsilon(t)\|_{2\beta}\|\nabla u^\epsilon(t)\|_{2\beta}
 \le C_\epsilon t^{-n+\frac{n}{2\beta}-\frac12}=\mathcal{O}(t^{-(1+\frac n2(1-\frac1\beta))})
\qquad\text{as $t\to+\infty$}.
\]
Therefore, for all $1\le q\le\infty$, both Part~(1) and Part~(2) of Theorem~\ref{prop:pro-lin} do apply.

Hence, recalling expression~\eqref{eq:lambda-u},
\begin{equation}
\label{pro:Phi-u}
 \biggl\| \Phi(\cdot,t) + M_\epsilon(\cdot,t)\vec\lambda_\epsilon  \,\biggr\|_q= o(t^{-\frac n2(1-\frac1q)}), 
  \qquad \text{with}\quad
      1\le q\le\infty.
 \end{equation}
Combining the integral formula~\eqref{abstr:eq} with the above results we get the asymptotics
\[
\biggl\|u^\epsilon(t)+M_\epsilon(\cdot,t)\vec\lambda_\epsilon  \,\biggr\|_q= o(t^{-\frac n2(1-\frac1q)}),
 \qquad
      1\le q\le\infty.
\]
On the other hand, from~\eqref{ideM1}.
\[
M_\epsilon(\cdot,t)\vec\lambda_\epsilon=e^{t\Delta}\vec\lambda_\epsilon=E(\cdot,t)\vec\lambda_\epsilon.
\]
and we readily obtain profile~\eqref{eq:pro-u}.
To deduce from~\eqref{eq:pro-u}
the upper and lower bounds~\eqref{eq:ul-bounds-u} we just write
\[ 
\|u^\epsilon(t)\|_q=\|E(\cdot,t)\vec\lambda_\epsilon\|_q\, t^{-n/2(1-1/q)} +o(t^{-n/2(1-1/q)})
 \qquad\text{as $t\to+\infty$}.
\]
and observe that there exist two constants 
$c_\epsilon,c'_\epsilon>0$ such that, for all $1\le q\le \infty$,
\[ c'_\epsilon|\vec\lambda_\epsilon|\, t^{-n/2(1-1/q)} 
\le\|E(\cdot,t)\vec\lambda_\epsilon\|_q
\le c'_\epsilon|\vec\lambda_\epsilon|\, t^{-n/2(1-1/q)}.\]
Estimates~\eqref{eq:ul-bounds-u} now follows.

The asymptotic profile~\eqref{eq:pro-u} and its corollary~\eqref{eq:ul-bounds-u} remain valid
without restrictions on the spatial dimension, provided $u_0\in L^1\cap L^n(\R^n)$,
with $\|u_0\|_{\dot B^{-1+n/p}_{p,\infty}}$
small enough for some $n<p<2n$. Indeed, in this case we can directly apply 
Proposition~\eqref{prop:l1l3} and obtain the validity of estimates~\eqref{decay-st} for all $t>0$
(in fact, $u^\epsilon\cap \mathcal{Y}_1\cap \mathcal{Y}_\infty$ that provides a better control of
$u^\epsilon$ near $t=0$ than~\eqref{decay-st}) and argue as above.
 
\end{proof}

Let us mention that in Theorem~\ref{th:main} it would be possible to replace the condition $u_0\in L^1(\R^n)$ (that implies $\int u_0=0$ because $u_0$ is divergence-free) by a more general 
condition prescribing a fast $L^q$-decay of the linear part of the equation.
For example, by the condition 
\begin{equation}
\label{eq:fadec}
\lim_{t\to+\infty}t^{(n/2)(1-1/q)}\|e^{t\Delta}u_0\|_q=0.
\end{equation}
Conclusions~\eqref{eq:pro-u}-\eqref{eq:ul-bounds-u} would remain valid.
However, formula~\eqref{liminte} does require the spatial integrability of the solution and cannot be used
in these more general situations. 
In the absence of the $L^1$ condition for $u_0$, formula~\eqref{eq:lambda-u} can be used
instead of~\eqref{liminte} to define $\vec\lambda_\epsilon$.

It could be also possible to replace the $L^1$-condition on~$u_0$ by a condition
involving the \emph{decay character} of $u_0$, namely, $r^*(u_0)>0$.
The decay character is very useful to get sharp algebraic decay estimates from below and above
for (linear or nonlinear) dissipative systems.
We do not recall here the precise definition of decay character of an $L^2$ function: the original definition~\cite{BjoS09,NicS15} has been slightly changed and improved in~\cite{Bra16},
in order to make the theory more complete and 
widely applicable. See also~\cite{FerNP17} for another very recent application of the decay character.

The case $r^*(u_0)<0$ is also of interest, but  this case
corresponds to solutions such that the linear part decays at slow rates in $L^2$:
as we pointed out in the introduction, in this case,
the analysis of~\cite{NicS15} already provides a satisfactory answer to the large time decay problem for equation~\eqref{eq:comp-appr}.

The case $r^*(u_0)=0$ corresponds to a borderline situation: this condition
ensures $ct^{-n/4}\le \|e^{t\Delta}u_0\|_2\le c't^{-n/4}$ for  $t>\!\!\!>1$ (with $c,c'>0)$.
As $u_0$ is divergence-free, condition $r^*(u_0)=0$ thus excludes 
$u_0\in L^1(\R^n)$ and excludes also~\eqref{eq:fadec} (but condition $r^*(u_0)=0$
is compatible, e.g., with $u_0\in \dot B^{-n/2}_{2,\infty}(\R^n)$). 
In this sitation, if $\vec\lambda_\epsilon=0$ then the nonlinear integral term decays
faster to zero than $\|e^{t\Delta}u_0\|_2$. So the linear part will govern the long time
behavior of the solution. But for $\vec\lambda_\epsilon\not=0$, both the linear and nonlinear terms have the same decay rates. In this borderline situation, solutions to~\eqref{eq:comp-appr} satisfy $\|u(t)\|_2=\mathcal{O}(t^{-n/4})$ but lower bounds for the $L^2$-decay of~$u$ are no longer available.

\section{Acknowledgements}

The author would like to thank the referees for their careful reading and helpful suggestions in improving the first version of the manuscript.

%
%
%
%
%
%
%
%
%
%
%

\end{document}